\documentclass[11pt]{amsproc}
\usepackage[utf8]{inputenc}
\usepackage{amsthm}
\usepackage{amsmath}
\usepackage{amssymb}
\usepackage{comment}
\usepackage[margin=1.5in]{geometry}
\usepackage{comment}
\usepackage[dvipsnames]{xcolor}
\usepackage{tikz}
\usepackage{color}
\usepackage[colorlinks=true, pdfstartview=FitV, linkcolor=blue, citecolor=blue, urlcolor=blue]{hyperref}

\newcommand\R{\mathbb{R}}

\newcommand\Z{\mathbb{Z}}

\newtheorem{Thm}{Theorem}[section]
\newtheorem{Lemma}[Thm]{Lemma}

\newtheorem{Con}[Thm]{Conjecture}
\newtheorem{Prop}[Thm]{Proposition}

\theoremstyle{remark}

\renewcommand\hat{\widehat}
\title{Note on Robins' Conjecture in Dimension Four and Higher}
\author{Oleg Asipchuk}
\address{Oleg Asipchuk,  University of Cincinnati,
	Department of Mathematical Sciences,
	Cincinnati, OH 45221, USA}
\email{asipchah@ucmail.uc.edu}


\begin{document}
 
\begin{abstract}
	This article is motivated by a conjecture proposed by Sinai Robins in 2024. The conjecture asserts that two convex, centrally symmetric sets of positive measure that are not multi-tilers must coincide up to rigid motions if and only if their Fourier transforms agree on the lattice $\mathbb{Z}^d$. In this paper, we disprove the conjecture by constructing explicit counterexamples in dimensions $d \geq 4$.
\end{abstract}

\maketitle

\section{Introduction}

For a function $f \in L^1(\mathbb{R}^d)$, its {\it Fourier transform} is defined by
\[
\hat{f}(\xi) = \int_{\mathbb{R}^d} f(x)\, e^{-2\pi i \langle x, \xi \rangle} \, dx, 
\quad \xi \in \mathbb{R}^d.
\]
If $f \in L^2(\mathbb{R}^d)$, the Fourier transform extends by density to a unitary operator on $L^2(\mathbb{R}^d)$. 

The Fourier transform is a fundamental mathematical tool that allows us to analyze functions and signals in terms of their frequency components. For further details, see \cite{Bbracewell99,Fourier,Stein2011}.

Let $A \subseteq \mathbb{R}^d$. The \emph{characteristic function} (or indicator function) of $A$ is defined by
\[
1_A(x) =
\begin{cases}
1, & \text{if } x \in A, \\
0, & \text{if } x \notin A.
\end{cases}
\]
Thus, $\hat{1}_A$ denotes the Fourier transform of the characteristic function of the set $A$.

The set $A \subset \mathbb{R}^d$ is said to \emph{$k$-tile} $\mathbb{R}^d$ with the translation set $\Lambda$ if almost every point in $\mathbb{R}^d$ is covered exactly $k$ times by the translates $A+\lambda$ for $\lambda \in \Lambda$, that is,
\begin{equation}\label{E-multi}
    \sum_{\lambda \in \Lambda} 1_{A}(x-\lambda) = k \quad \text{a.e. in } \mathbb{R}^d.
\end{equation}
A set with this property is called a \emph{multi-tiler}. In the special case $k=1$, we simply say that $A$ \emph{tiles} $\R^d$. Tilings and multi-tilings naturally arise in diverse areas such as quasicrystals, Fourier analysis, and the theory of lattices. For further details, see \cite{Agora2015,Fuglede1974,GrepstaLev2014,Iosevich2013,KolountzakisLev2021} and the references therein.

In this article, we construct a counterexample to the following conjecture.

\begin{Con}[Robins, 2024]\label{ConRobins}
    Let $\mathcal{P},\mathcal{Q}\subset \R^d$ be two convex, central symmetric bodies. Suppose $P,Q$ are not multi-tilers. Then $\hat{1}_\mathcal{P} (\xi)=\hat{1}_\mathcal{Q} (\xi)$ for all $\xi\in\Z^d$ if and only if $\mathcal{P}=\mathcal{Q}$ up to rigid motions, except on a set of measure zero.
\end{Con}

The main result of the article can be formulated as the following theorem.

\begin{Thm}\label{T-main}
    Conjecture \ref{ConRobins} is false for $d\geq 4$. 
\end{Thm}

\medskip

The paper is organized as follows. Section 2 presents the proof of Theorem \ref{T-main} together with related remarks, while Section 3 states and proves a result illustrating the importance of convexity in Conjecture \ref{ConRobins}. 

\medskip
\noindent
{\it Acknowledgments}. The author would like to acknowledge Sinai Robins for valuable mathematical discussions and remarks.

\section{Proof of the main result} 
First, we recall a well-known result, which can be found, for example, in \cite[Lecture 2]{Kolountzakis2004}.
\begin{Lemma}\label{L-Kolountzakis}
Let $Q \subset \mathbb{R}^d$ be a measurable set of positive and finite measure, and let $\Lambda \subset \mathbb{R}^d$ be a full-rank lattice. Then
\[
Q \text{ tiles } \mathbb{R}^d \text{ by translations with } \Lambda
\quad \Longleftrightarrow \quad
\hat{1}_Q(\xi) = 0 \text{ for all } \xi \in \Lambda^{*} \setminus \{0\},
\]
where $\Lambda^{*} = \{\, \xi \in \mathbb{R}^d : \xi \cdot \lambda \in \mathbb{Z} \text{ for all } \lambda \in \Lambda \,\}$ is the dual lattice.
\end{Lemma}

Let $R$ be the rhombus with vertices  
\begin{align}\label{E-R}
    (-1,0),\ (0,3),\ (1,0),\ \text{and } (0,-3),
\end{align}
and let $H$ be the hexagon with vertices  
\begin{align}\label{E-H}
   (-1,1),\ (0,2),\ (1,1),\ (1,-1),\ (0,-2),\ \text{and } (-1,-1).
\end{align}

Before proceeding to the proof of the main theorem, we state the following proposition.
\begin{Prop}\label{P-Calculus}
    \begin{equation}\label{Eq-=}
    \hat{1}_R(\xi)=\hat{1}_H(\xi)=\begin{cases}
        6 & \text{if } \xi=(0,0) \\
        0 & \text{if } \xi \neq (0,0) 
    \end{cases}
\end{equation}
for all $\xi\in\Z^2$. 
\end{Prop}
\begin{proof}
   Both $R$ and $H$ tile $\mathbb{R}^2$ by translations with the lattice
\[
\Lambda = \bigcup_{j=0}^{1} \{ (2n + j,\; 6m + 3j) : n,m \in \mathbb{Z} \}.
\]
So, by Lemma \ref{L-Kolountzakis} $\hat{1}_Q(\xi) = 0$  for all $\xi \in \Lambda^{*} \setminus \{0\}$, where 
\[
\Lambda^* = \bigcup_{j=0}^{1} \{ (n + j/2,\; m/3 + j/6) : n,m \in \mathbb{Z} \}.
\]
$\Z^2\subset \Lambda^*$, so $\xi \in \Z^2 \setminus \{0\}$. 

For $\xi = (0,0)$, a direct computation gives
\[
\widehat{1}_R(0,0) = \int_R dx = 6 = \int_H dx = \widehat{1}_H(0,0).
\]

\end{proof}

Now, we are ready to prove the main result.

\begin{proof}[Proof of Theorem \ref{T-main}]

Let $d\geq 4$. We consider two sets in $\R^2$, rhombus $R$ with vertices \eqref{E-R} and hexagon $H$ with vertices \eqref{E-H}. Both $R$ and $H$ are centrally symmetric and convex.  Moreover, the Fourier transforms of the characteristic functions of $R$ and $H$ satisfy \eqref{Eq-=}.

Next, let $B \subset \mathbb{R}^{d-2}$ denote the unit ball centered at the origin. 
We define two sets: $\mathcal{P} = R \times B$ and $\mathcal{Q} = H \times B$. 
Since they are Cartesian products of centrally symmetric convex sets, both $\mathcal{P}$ and $\mathcal{Q}$ are themselves centrally symmetric and convex.

Note that if a convex body $P$ admits a $k$-tiling, then $P$ must be a polytope \cite[Introduction]{Gravin2012}. Since neither $\mathcal{P}$ nor $\mathcal{Q}$ is a polytope, they cannot be multi-tilers.

Using \eqref{Eq-=}, we see that for all $\xi_1 \in \mathbb{Z}^2$ and $\xi_2 \in \mathbb{Z}^{d-2}$, with $\xi = (\xi_1, \xi_2)$,
\begin{align*}
 \hat{1}_{\mathcal{P}}(\xi) = \hat{1}_{R}(\xi_1)\hat{1}_{B}(\xi_2) =  \hat{1}_{H}(\xi_1)\hat{1}_{B}(\xi_2) =  \hat{1}_{\mathcal{Q}}(\xi).
\end{align*}
Therefore, the conjecture is false for $d\geq 4$.
\end{proof}

\subsection{Remarks}
\begin{enumerate}
    \item Conjecture \ref{ConRobins} has not appeared in the published literature, but it has been stated by Sinai Robins in several of his talks. A revised version of the conjecture can be found in \cite{Robins2026}.
    \item The proof of Theorem \ref{T-main} remains valid if $H$ and $R$ are replaced by any other multi-tilers whose Fourier transforms agree on the lattice $\mathbb{Z}^2$.
    \item The method described in Section 3 does not apply for $d < 4$, since in this case one of the sets must be one-dimensional. In dimension $1$, however, the only convex and centrally symmetric set is an interval, which trivially multi-tiles $\mathbb{R}$. Moreover, since $|I| = \hat{1}_I(0)$ for any interval $I \subset \mathbb{R}$, it is impossible to find two distinct intervals (up to rigid motions) whose Fourier transforms agree on the lattice $\mathbb{Z}$.
\end{enumerate}

\section{Counterexample to the conjecture without the Convexity Assumption.} 

\begin{Thm}\label{T-main1}
    Conjecture \ref{ConRobins} without the Convexity Assumption is false for $d\geq 2$. 
\end{Thm}

\begin{proof}
    We consider two sets in $\R$: $I=[0,1]$ and $J=[-1,-1/2)\cup[1/2,1]$. Note that $I$ and $J$ have measure $1$ and tile $\R$ with the same lattice $\Z$. So, by Lemma \ref{L-Kolountzakis} we conclude that   
    \begin{equation}\label{Eq-=2}
    \hat{1}_I(\xi)=\hat{1}_J(\xi)=\begin{cases}
        1 & \text{if } \xi=0 \\
        0 & \text{if } \xi \neq 0
    \end{cases}
    \end{equation}
    for all $\xi\in\Z$. 

    Next, let $B \subset \mathbb{R}^{d-1}$ denote the unit ball centered at the origin. We define two sets: $\mathcal{P} = I \times B$ and $\mathcal{Q} = J \times B$.  Since they are Cartesian products of centrally symmetric sets, both $\mathcal{P}$ and $\mathcal{Q}$ are themselves centrally symmetric. Note that $\mathcal{P}$ is a ball and $\mathcal{Q}$ is the union of two disjoint balls in $\mathbb{R}^d$, so neither of them can be a multi-tiler.

Using \eqref{Eq-=2}, we see that for all $\xi_1 \in \mathbb{Z}$ and $\xi_2 \in \mathbb{Z}^{d-1}$, with $\xi = (\xi_1, \xi_2)$,
\begin{align*}
 \hat{1}_{\mathcal{P}}(\xi) = \hat{1}_{I}(\xi_1)\hat{1}_{B}(\xi_2) =  \hat{1}_{J}(\xi_1)\hat{1}_{B}(\xi_2) =  \hat{1}_{\mathcal{Q}}(\xi).
\end{align*}
Therefore, the conjecture without the Convexity Assumption is false for $d\geq 2$.
\end{proof}

\bibliographystyle{plain}
\bibliography{Main}

@book {Fourier,
   AUTHOR = {Fourier, J.},
     TITLE = {The {A}nalytical {T}heory of {H}eat},
 PUBLISHER = {G. E. Stechert \& Co., New York},
      YEAR = {1945},
     PAGES = {xxiii+466},
   MRCLASS = {36.0X},
  MRNUMBER = {11523},
}

@book{Stein2011,
  title={Fourier Analysis: An Introduction},
  author={Stein, E.M. and Shakarchi, R.},
  isbn={9781400831234},
  series={Princeton lectures in analysis},
  url={https://books.google.com/books?id=FAOc24bTfGkC},
  year={2011},
  publisher={Princeton University Press}
}

@book {Bbracewell99,
    AUTHOR = {Bracewell, R. N.},
     TITLE = {The {F}ourier transform and its applications},
    SERIES = {McGraw-Hill Series in Electrical Engineering. Circuits and
              Systems},
   EDITION = {Third},
 PUBLISHER = {McGraw-Hill Book Co., New York},
      YEAR = {1986},
     PAGES = {xx+474},
      ISBN = {0-07-007015-6},
   MRCLASS = {42A38},
  MRNUMBER = {924577},
}

@book{Robins2026,
    AUTHOR = {S. Robins},
     TITLE = {Fourier analysis on polytopes, part II: {A}pplications of {P}oisson summation},
    SERIES = {AMS Student Mathematical Library series},
 PUBLISHER = {American Mathematical Society},
      YEAR = {2026},
}

@article {Fuglede1974,
    AUTHOR = {Fuglede, B.},
     TITLE = {Commuting self-adjoint partial differential operators and a
              group theoretic problem},
   JOURNAL = {Journal of Functional Analysis},
  FJOURNAL = {Journal of Functional Analysis},
    VOLUME = {16},
      YEAR = {1974},
     PAGES = {101--121},
      ISSN = {0022-1236},
   MRCLASS = {47F05 (81.47)},
  MRNUMBER = {470754},
       DOI = {10.1016/0022-1236(74)90072-x},
       URL = {https://doi.org/10.1016/0022-1236(74)90072-x},
}

@article {GrepstaLev2014,
    AUTHOR = {Grepstad, S. and Lev, N.},
     TITLE = {Multi-tiling and {R}iesz bases},
   JOURNAL = {Advances in Mathematics},
  FJOURNAL = {Advances in Mathematics},
    VOLUME = {252},
      YEAR = {2014},
     PAGES = {1--6},
      ISSN = {0001-8708},
   MRCLASS = {52C23 (42B10)},
  MRNUMBER = {3144222},
MRREVIEWER = {Jordi Marzo},
       DOI = {10.1016/j.aim.2013.10.019},
       URL = {https://doi.org/10.1016/j.aim.2013.10.019},
}

@article {KolountzakisLev2021,
   AUTHOR = {Kolountzakis, Mihail N. and Lev, Nir},
     TITLE = {Tiling by translates of a function: results and open problems},
   JOURNAL = {Discrete Anal.},
  FJOURNAL = {Discrete Analysis},
      YEAR = {2021},
     PAGES = {Paper No. 12, 24},
   MRCLASS = {42A38 (52C22)},
  MRNUMBER = {4317275},
MRREVIEWER = {Chun-Kit Lai},
       DOI = {10.19086/da.28122},
       URL = {https://doi.org/10.19086/da.28122},
}

@article {Gravin2012,
    AUTHOR = {Gravin, N. and Robins, S. and Shiryaev, D.},
     TITLE = {Translational tilings by a polytope, with multiplicity},
   JOURNAL = {Combinatorica},
  FJOURNAL = {Combinatorica. An International Journal on Combinatorics and
              the Theory of Computing},
    VOLUME = {32},
      YEAR = {2012},
    NUMBER = {6},
     PAGES = {629--649},
      ISSN = {0209-9683},
   MRCLASS = {52C22},
  MRNUMBER = {3063154},
MRREVIEWER = {Elizaveta Zamorzaeva},
       DOI = {10.1007/s00493-012-2860-3},
       URL = {https://doi.org/10.1007/s00493-012-2860-3},
}

@article {Agora2015,
    AUTHOR = {Agora, E. and Antezana, J. and Cabrelli, C.},
     TITLE = {Multi-tiling sets, {R}iesz bases, and sampling near the
              critical density in {LCA} groups},
   JOURNAL = {Advances in Mathematics},
  FJOURNAL = {Advances in Mathematics},
    VOLUME = {285},
      YEAR = {2015},
     PAGES = {454--477},
      ISSN = {0001-8708},
   MRCLASS = {94A20 (22B05)},
  MRNUMBER = {3406506},
MRREVIEWER = {Yunus Emre Yildirir},
       DOI = {10.1016/j.aim.2015.08.006},
       URL = {https://doi.org/10.1016/j.aim.2015.08.006},
}

@article {Iosevich2013,
    AUTHOR = {Iosevich, A. and Kolountzakis, M. N.},
     TITLE = {Periodicity of the spectrum in dimension one},
   JOURNAL = {Analysis and PDE},
  FJOURNAL = {Analysis \& PDE},
    VOLUME = {6},
      YEAR = {2013},
    NUMBER = {4},
     PAGES = {819--827},
      ISSN = {2157-5045},
   MRCLASS = {42B05},
  MRNUMBER = {3092730},
MRREVIEWER = {Kasso A. Okoudjou},
       DOI = {10.2140/apde.2013.6.819},
       URL = {https://doi.org/10.2140/apde.2013.6.819},
}

@incollection {Kolountzakis2004,
    AUTHOR = {Kolountzakis, M. N.},
     TITLE = {The study of translational tiling with {F}ourier analysis},
 BOOKTITLE = {Fourier analysis and convexity},
    SERIES = {Applied and Computational Harmonic Analysis},
     PAGES = {131--187},
 PUBLISHER = {Birkh\"{a}user Boston, Boston, MA},
      YEAR = {2004},
   MRCLASS = {42B35 (05B45 42-02 52C22)},
  MRNUMBER = {2087242},
}

\end{document}